\documentclass[11pt]{amsart}
\usepackage{amsmath,amssymb,amsbsy,amsfonts,amsthm,latexsym,
            amsopn,amstext,amsxtra,euscript,amscd,color}
\PassOptionsToPackage{hyphens}{url}\usepackage{hyperref}
 
 \usepackage[capbesideposition=outside,capbesidesep=quad]{floatrow}
 \usepackage[normalem]{ulem}

\restylefloat{table}
\restylefloat{table}
            
\usepackage{multirow,caption}

\usepackage{amscd}
\usepackage{color,enumerate}

\usepackage[colorinlistoftodos,prependcaption,textsize=tiny]{todonotes}

\newtheorem{thm}{Theorem}[section]
\newtheorem{lem}[thm]{Lemma}

\newtheorem{thm-con}[thm]{Theorem-Conjecture}
\numberwithin{equation}{section}

\theoremstyle{definition}

\newcommand{\f}{\mathbb F}

\begin{document}

\title[Dembowski-Ostrom polynomials and Dickson polynomials]{Dembowski-Ostrom polynomials and Dickson polynomials}

\author[S. U. Hasan]{Sartaj Ul Hasan}
\address{Department of Mathematics, Indian Institute of Technology Jammu, Jammu 181221, India}
\email{sartaj.hasan@iitjammu.ac.in}

\author[M. Pal]{Mohit Pal}
\address{Department of Mathematics, Indian Institute of Technology Jammu, Jammu 181221, India}
\email{2018RMA0021@iitjammu.ac.in}

\begin{abstract}
We give a complete classification of Dembowski-Ostrom polynomials from the composition of Dickson polynomials of arbitrary kind and monomials over finite fields. Moreover, by using a variant of the Weil bound for the number of points of affine algebraic curves over finite fields, we discuss the planarity of the obtained Dembowski-Ostrom polynomials. Dembowski-Ostrom polynomials play a crucial role in coding theory and cryptography.
\end{abstract}



\maketitle

\section{Introduction}\label{S1}
Let $p$ be an odd prime and $e$ be a positive integer. We denote by $\f_q$ the finite fields with $q =p^e$ elements, and by $\f_q^*$ the multiplicative group of non-zero elements of $\f_q$. Let $f$ be a function from finite field $\f_q$ to itself. Since any function from a finite field to itself can be uniquely represented by a polynomial of degree at most $q-1$, we shall often consider $f$ to be a polynomial in $\f_q[X]$. A polynomial $f$ is called permutation polynomial if it permutes the elements of the finite fields. Permutation polynomials have been studied widely in the past decades due to their wide range of applications in coding theory and cryptography. One may refer to \cite {Hou} for a recent survey on permutation polynomials. 

For any non-negative integer $k$, the $k$-th Dickson polynomial of the first and second kind are given by
\begin{equation*}
  D_k(X,a) :=\sum_{i=0}^{\lfloor \frac{k}{2} \rfloor} \frac{k}{k-i} {k-i \choose i} (-a)^i X^{k-2i}
\end{equation*}
and
\begin{equation*}
E_k(X,a) :=\sum_{i=0}^{\lfloor \frac{k}{2} \rfloor} {k-i \choose i} (-a)^i X^{k-2i},                                                                                 
\end{equation*}
respectively, where $a \in \mathbb{F}_q$ is a parameter. Dickson polynomial of the first kind was introduced by Dickson \cite{LED} in his Ph.D. thesis of 1897 and its variant, which is now called Dickson polynomial of the second kind, was introduced by Schur \cite{IS} more than two decades later in 1923. It may be noted that these polynomials are closely related to the classical Chebyshev polynomials.
Dickson polynomials of the first and second kind have been studied extensively, especially in the context of their permutation behaviour. We refer interested readers to the monograph \cite{Dickson-book} for more details about Dickson polynomials. One may also refer to the handbook \cite{HB} for some recent progress about Dickson polynomials and related objects. A trigonometric approach for Dickson polynomials has been recently considered by Lima and Panario \cite{LP}. The Dickson polynomials have also been used to study the $c$-differential uniformity of some functions over finite fields~\cite{HPRS}. 

In 2013, Wang and Yucas \cite{WY} defined Dickson polynomial of the $(m+1)$-th kind, where $m$ is a non-negative integer, as follows
\begin{equation}
\label{domk}
 D_{k,m}(X,a) :=\sum_{i=0}^{\lfloor \frac{k}{2} \rfloor} \frac{k-mi}{k-i} {k-i \choose i} (-a)^i X^{k-2i},
\end{equation}
where $a \in \mathbb{F}_q$ is a parameter and $D_{0,m}(X,a)=2-m$ and $D_{1,m}(X,a)=X.$ It is clear from the definition that $D_{k,0}(X,a)= D_k(X,a)$ and $D_{k,1}(X,a)=E_k(X,a)$. In \cite{WY}, the authors also gave a recurrence relation,
\begin{equation}
\label{rr}
D_{k,m}(X,a) = mE_{k}(X,a)-(m-1)D_{k}(X,a)
\end{equation}
to find Dickson polynomial of the $(m+1)$-th kind from Dickson polynomial of the first and second kind. It is easy to see from (\ref{rr}) that over any finite field $\mathbb{F}_q$ of characteristic $p$, $D_{k,m+p}(X,a) = D_{k,m}(X,a)$ and thus, we shall restrict ourselves to the case $m<p$.

Let $f$ be a function from $\mathbb{F}_{p^e}$ to itself. For any $\epsilon \in \f_{p^e}$, let $\Delta_f(X, \epsilon) := f(X+\epsilon)-f(X)-f(\epsilon)$. A function $f$ is said to be planar if the corresponding function $\Delta_f(X,\epsilon)$ is a permutation of $\f_{p^e}$ for all $\epsilon \in \f_{p^e}^*$. It is straightforward to see that in the case of characteristic two, $X+\epsilon$ and $X$ have the same image under $\Delta_f$ and hence there is no planar function over a finite field of even characteristic. Planar functions are very important objects as they have received plethora of applications in coding theory and cryptography among others. A function $f$ is called exceptional planar if it is planar over $\f_{p^e}$ for infinitely many $e$. The problem of classifying planar functions from a given class of polynomials is not easy in general. For instance, it was a long standing problem to classify planar functions from the monomials over finite fields of odd characteristic, and eventually was resolved recently by Zieve~\cite{MZ}.

A {\it Dembowski-Ostrom} (DO) polynomial over a finite field $\f_{p^e}$ is a polynomial that admits the following shape
\begin{equation*}
\displaystyle \sum_{i,j} a_{ij}X^{p^i+p^j},
\end{equation*}
where $a_{ij} \in \f_{p^e}$. These polynomials were introduced by Dembowski and Ostrom \cite{DO} to study some projective planes. However, these polynomials later turned out to be a rich source of planar functions. In fact, all the known exceptional planar function are DO polynomials except for the monomial $X^{\frac{3^{\alpha}+1}{2}}$ over $\mathbb{F}_{3^e}$ with $\gcd(e, \alpha)=1$ and $e$ odd. Planar DO polynomials are closely related with Dickson polynomials of the first kind. One may note that all the known exceptional planar DO polynomials were actually stemmed from Dickson polynomials. This observation inspired Coulter and Matthews \cite{CM} to classify DO polynomials from Dickson polynomials of first and second kind over finite fields of odd characteristic. Recently, the authors along with Fernando \cite{FHP} classified DO polynomials from the composition of reversed Dickson polynomials of arbitrary kind and monomials over finite fields of odd characteristic.

In continuation of our work in \cite{FHP}, we now consider the problem of classifying DO polynomials stemming from the Dickson polynomials of arbitrary kind over finite fields of odd characteristic. In fact, we give a complete classification of DO polynomials arising from the composition of Dickson polynomial of the $(m+1)$-th kind (where $m\geq2$) and the monomial $X^d$, where $d$ is a positive integer, in odd characteristic. The motivation behind considering this composition actually stems from the fact that the exceptional planar polynomials $X^{10} \pm X^6 -X^2$ are essentially the composition of the Dickson polynomials $D_5(X, \pm1)$ and
the monomial $X^2$. The DO polynomials do not have any constant term. Therefore, throughout this paper, we shall consider the polynomial $D_{k,m}(X^d,a)-D_{k,m}(0,a)$, which, in short, shall be denoted by $\mathfrak D_{k,m}$. Here $D_{k,m}(0,a)$ is the constant term which arises only when $k$ is even.

We now give the structure of the paper. Section \ref{S2}, \ref{S3} and \ref{S4} are devoted to the classification of DO polynomials from the polynomials $\mathfrak D_{k,2}$, $\mathfrak D_{k,3}$ and $\mathfrak D_{k,4}$, respectively. By using the properties of the Dickson polynomials of arbitrary kind and the characteristic of the underlying finite field, the classification of
DO polynomials from $\mathfrak D_{k,m}$, when $m\geq 5$, have been considered in Section~\ref{S5}. Section \ref{S6} is devoted to the planarity aspect of all the DO polynomials obtained from the previous sections which we have listed in the Appendix \ref{A1}. 

Throughout the paper, we always assume that $n, s,t, \alpha, \beta, \gamma, \delta$ are nonnegative integers unless specified otherwise.


\section{DO polynomials from the polynomial $\mathfrak{D}_{k,2}$} \label{S2}
As alluded to in Introduction, the classification of DO polynomials arising from Dickson polynomials of the first and second kind has already been considered in \cite{CM}. Therefore, in this section, we shall consider classification of DO polynomials arising from the Dickson polynomials of the third kind. 

We know that DO polynomials are closed under the left or right composition with $X^p$, therefore, it is sufficient to consider the cases when $(d,p)=1$. One may also note that the monomial $X^{rd}$,  where $r$ is positive integer, is DO polynomial if and only if $rd=p^{\beta}(p^{\alpha }+1)$, where $\alpha, \beta$ are nonnegative integers and $\beta$ is the highest exponent of $p$ such that $p^{\beta} \mid r$. It is clear that whenever $(r,p)=1$, we must have $\beta=0$. Also notice that $D_{k,m}(X,0)=X^k$, which is a DO polynomial if and only if $k$ is of the form $p^{\beta}(p^{\alpha}+1)$. Therefore in what follows, we shall always consider $a \in \mathbb{F}_{q}^*$. These assumptions and conventions shall remain effective across the Sections.

Recall that $\mathfrak{D}_{k,2} =D_{k,2}(X^d,a)-D_{k,2}(0,a)$, whose expression is given by
\begin{equation*}
 \mathfrak D_{k,2} =\sum_{i=0}^{\lfloor \frac{k-1}{2} \rfloor} \frac{k-2i}{k-i} {k-i \choose i} (-a)^i X^{(k-2i)d},
\end{equation*}
i.e.,
\begin{equation} \label{fk}
 \mathfrak D_{k,2}=X^{kd}-(k-2)aX^{(k-2)d}+\frac{(k-3)(k-4)}{2!}a^2X^{(k-4)d}-  \cdots.
\end{equation}
The following theorem gives the conditions on $k,d$ and $p$ for which the polynomial $\mathfrak D_{k,2}$ turns out to be a DO polynomial.
\begin{thm} 
Let $q$ be a power of odd prime $p$ and $a\in\mathbb{F}_q^*$. The polynomial $\mathfrak D_{k,2}$ is DO polynomial over $\mathbb{F}_q$ if and only if one of the following holds.
\begin{enumerate}[(i)]
\item $k=1$ and $d= p^n(p^{\alpha}+1)$ for non-negative integers $\alpha$ and $n$.
\item $k=2$ and $d= p^n(p^{\alpha}+1)/2$ for non-negative integers $\alpha$ and $n$.
\item $k=3$ and either
\begin{enumerate}
 \item $p=3$ and $d= p^n(p^{\alpha}+1)$ for non-negative integers $\alpha$ and $n$; or
 \item $p=5$ and $d= 2p^n$ for non-negative integer $n$.
\end{enumerate}
\item $k=4,6,12$, $p=3$ and $d= p^n$ for non-negative integer $n$.
\item $k=5$ and either
\begin{enumerate}
\item $p=3$ and $d= 2p^n$ for non-negative integer $n$; or
\item $p=5$ and $d= 2p^n$ for non-negative integer $n$.
\end{enumerate}

\item $k=9$, $p=3$ and $d= 4p^n$ for non-negative integer $n$.

\end{enumerate}
\end{thm}

\begin{proof} It is easy to see that each of the cases listed in the theorem correspond to the DO polynomials. Now it only remains to show that all the DO polynomials obtained from the polynomial $\mathfrak D_{k,2}$ shall fall in any one of the above mentioned cases. 

We begin with the polynomial $\mathfrak D_{1,2}=X^d$ and it will be a DO polynomial only if $d$ is of the form $p^{\alpha}+1$. The polynomial $\mathfrak D_{2,2}=X^{2d}$ will be a DO polynomial only if $d$ is of the form $ (p^{\alpha}+1)/2$. If the polynomial $\mathfrak D_{3,2}= X^{3d}-aX^{d}$ is a DO polynomial then $d = p^{\alpha}+1$ and $3d = p^s(p^{\beta}+1)$. Now $p^s \mid 3$ implies that either $p=3$ and $s=1$ or $p>3$ and $s=0$. In the event of $p=3$, $\mathfrak D_{3,2}$ will be DO polynomial only if $d$ is of the form $p^{\alpha}+1$. However, in case of $p>3$, we have $d = p^{\alpha}+1$ and $3d = p^{\beta}+1$ and these equations would lead to $3p^{\alpha}+2 = p^{\beta}$, which is true if and only if $\alpha = 0$, $\beta=1$ $p = 5$ and $d=2$. The polynomial $\mathfrak D_{4,2}=X^{4d}-2aX^{2d}$ will be a DO polynomial only if $2d = p^{\alpha}+1$ and $4d = p^{\beta}+1$. Combining these two equations, we get $2p^{\alpha}+1 = p^{\beta}$, which is true if and only if $\alpha = 0$, $\beta=1$,  $p = 3$ and $d=1$. For the polynomial $\mathfrak D_{5,2}=X^{5d}-3aX^{3d}+a^2X^{d}$, we are going to have two distinct cases, namely, $p=3$ and $p>3$. If $p=3$ then the coefficients of $X^{5d}$ and $X^d$ are nonzero. Thus, if $\mathfrak D_{5,2}$ is a DO polynomial then $d = 3^{\alpha}+1$ and $5d = 3^{\beta}+1$. Combining these two equations, we have $5\cdot 3^{\alpha}+4 = 3^{\beta}$, which forces $\alpha = 0$, $\beta=2$ and $d=2$. If $p>3$ then the coefficients of $X^{5d}$, $X^{3d}$ and $X^d$ are nonzero. Therefore, $\mathfrak D_{5,2}$ will be a DO polynomial only if $d = p^{\alpha}+1$, $3d = p^{\beta}+1$ and $5d = p^s(p^{\gamma}+1)$. Combining first two equations, we get $3p^{\alpha}+2 = p^{\beta}$, which is true if and only if $\alpha = 0$, $\beta=1$, $p=5$ and $d=2$.  Now putting these values into third equation, we have $10=5^s(5^{\gamma}+1)$, which gives $s=1$ and $\gamma=0$. Similarly, for the polynomial $\mathfrak D_{6,2}=X^{6d}-4aX^{4d}+3a^2X^{2d}$, we have two cases. If $p=3$ then the coefficients of $X^{6d}$ and $X^{4d}$ are nonzero. Thus, if $\mathfrak D_{6,2}$ is a DO polynomial then $4d = 3^{\alpha}+1$ and $\displaystyle 6d = 3^s(3^{\beta}+1)$. Since $3^s \mid 6$, we have $s=1$ and hence, second equation reduces to $2d =3^{\beta}+1$. Combining these two equations, we get $2\cdot3^{\beta}+1 = 3^{\alpha}$, which forces $\beta = 0$, $\alpha=1$ and $d=1$. If $p>3$ then the coefficients of $X^{2d}$, $X^{4d}$ and $X^{6d}$ are nonzero. In this case, $\mathfrak D_{6,2}$ will be a DO polynomial only if $2d = p^{\alpha}+1$, $4d = p^{\beta}+1$ and $6d=p^{\gamma}+1$. Combining first two equations, we get $2p^{\alpha}+1 = p^{\beta}$, which forces $\alpha = 0$. Thus, we get $p^{\beta}=3$, which is a contradiction to our assumption that $p>3$. For $k\geq 7$, consider the following cases.

\textbf{Case 1} Let $k \not\equiv 0,2~(\mbox{mod}~p)$. In this case, the coefficients of $X^{kd}$ and $X^{(k-2)d}$ in $\mathfrak D_{k,2}$ are nonzero. Therefore, if $\mathfrak D_{k,2}$ is a DO polynomial then $kd = p^{\alpha}+1$ and $(k-2)d = p^{\beta}+1$ and hence, $2d = p^{\alpha} - p^{\beta}$. Since $k \geq 7$, we have $\alpha, \beta \geq 1$, which implies that $p\mid d$. This gives us contradiction to the fact that $(p,d)=1$. Therefore, $\mathfrak D_{k,2}$ is not DO.

\textbf{Case 2} Let $k \equiv 2~(\mbox{mod}~p)$. In this case, coefficient of $X^{kd}$ in $\mathfrak D_{k,2}$ is nonzero. Now if $p=3$ then $k\equiv 2 \pmod 3$ implies that $3\nmid (k-3)$ and 
$3\nmid (k-4)$ and thus, the coefficient of $X^{(k-4)d}$ is nonzero. If $p>3$ then $k\equiv 2 \pmod p$ implies that $k\not\equiv 3,4 \pmod p$ and again, the coefficient of $X^{(k-4)d}$ is nonzero. Thus, if $\mathfrak D_{k,2}$ is a DO polynomial then $kd = p^{\alpha}+1$ and $(k-4)d = p^{\beta}+1$. Combining these two equations yields $4d = p^{\alpha}-p^{\beta}$. Since $k \geq 7$, we have $\alpha,\beta \geq 1$, which implies that $p \mid d$, a contradiction. Therefore, $\mathfrak D_{k,2}$ is not DO.

\textbf{Case 3} Let $k \equiv 0~(\mbox{mod}~p)$. Here we shall consider two subcases, namely, $p=3$ and $p>3$.

\textbf{Subcase 3.1} Let $p=3$. In this case, the polynomial $\mathfrak D_{9,2}=X^{9d}+2aX^{7d}+2a^3X^{3d}+a^4X^d$ will be a DO polynomial only if $9d = 3^s(3^{\alpha}+1)$, $7d = 3^{\beta}+1$, $3d=3^t(3^{\gamma}+1)$ and $d=3^{\delta}+1$. Since $3^s \mid 9$ and $3^t\mid 3$, we must have $s=2$ and $t=1$. Therefore, the first and third equation reduce to $d = 3^{\alpha}+1$ and $d=3^{\gamma}+1$. Combining first two equations, we get $3^{\beta} = 7\cdot3^{\alpha}+6$, which forces $\alpha = 1$, $ \beta = 3$ and $d = 4$. Similarly, if the polynomial $\mathfrak D_{12,2}=X^{12d}+2aX^{10d}+a^3X^{6d}+2a^4X^{4d}$ is a DO polynomial then $12d = 3^s(3^{\alpha}+1)$, $10d = 3^{\beta}+1$, $6d~=~3^t(3^{\gamma}+1)$ and $4d = 3^{\delta}+1$. Now $3^s \mid 12$ and $3^t \mid 6$ implies that $s=1$ and $t=1$, respectively. Therefore, the first and third equations reduce to $4d = 3^{\alpha}+1$ and $2d = 3^{\gamma}+1$, respectively. Now combining second and third equations, we get $3^{\beta} = 5\cdot 3^{\gamma}+4$, which forces $\gamma =0$ that, in turn, gives $d=1$, $\beta = 2$ and $\alpha = \delta = 1$. 

In the rest of this subsection, we shall prove that the polynomial $\mathfrak D_{k,2}$ is never DO whenever $k\geq15$ and $k\equiv0 \pmod 3$. Suppose on the contrary that $\mathfrak D_{k,2}$ is DO. Since $k \equiv 0~(\mbox{mod}~3)$, therefore $k \not\equiv 1,2(\mbox{mod}~3)$, hence the coefficients of $X^{kd}$ and $X^{(k-2)d}$ are nonzero. Since $\mathfrak D_{k,2}$ is DO by our assumption, therefore $kd=3^s(3^{\alpha}+1)$ and $(k-2)d=3^{\beta}+1,$ where $s$ is the largest exponent of $3$ which divides $k$. Notice that since $k\geq 15$, we must have $\beta \geq 3$, which implies that $d\equiv1 \pmod 3$. From first equation, since $3^s\mid k$, therefore, $k=u3^s$, for some positive integers $u$ and $s$. Putting this value of $k$ in the first equation, we have, $ud=3^{\alpha}+1$. We now consider two cases, namely, $\alpha =0$ and $\alpha \geq 1$. In the case $\alpha =0$, we have $ud=2$, therefore either $d=2$ and $u=1$ or $d=1$ and $u=2$. First case is not possible as $d\equiv1 \pmod 3$. In the case of $d=1$ and $u=2$, we have $k=2\cdot3^s$ and $k-2=3^{\beta}+1$. Combining these two equations, we have $2\cdot3^{s-1}-3^{\beta-1}=1$. Again we have two cases, namely, $s=1$ and $s\geq2$. If $s=1$, we have $3^{\beta-1}=1$ which is a contradiction as $\beta \geq 3$. In the case of $s\geq 2$, we have $2\cdot3^{s-1}-3^{\beta-1}=1$ and taking modulo $3$ on both sides , we get $0\equiv1 \pmod 3$, a contradiction. In the case $\alpha\geq 1$, equation $ud=3^{\alpha}+1$ implies that $ud\equiv 1 \pmod 3$ and since $d\equiv 1 \pmod 3$ therefore $u\equiv 1 \pmod 3$. We now show that in this case, $k\not \equiv 6 \pmod 9$. Recall that $k=u3^s$. If $u\equiv 1 \pmod 3$ then $u=3n+1$ for some non-negative integer $n$. If $n=0$, then $u=1$ and $k=3^s$. But $k\geq15$, we must have $s\geq3$, which implies that $k\equiv0 \pmod 9$ and hence, $k \not \equiv 6 \pmod 9$. In the case $n\geq1$, we have $$k=u3^s = (3n+1)3^s = n3^{s+1}+3^s.$$ If $s\geq2$ then $k\equiv 0 \pmod 9$ and if $s=1$ then $k=9n+3$, which implies that $k\equiv 3 \pmod 9$. Therefore $k\not \equiv 6 \pmod 9$.

We now show that the fifth term $$\frac{(k-5)(k-6)(k-7)(k-8)}{4!}a^4X^{(k-8)d}$$ will exist. Note that $k\equiv 0 \pmod 3$ implies that $3\nmid (k-5)$, $3\nmid (k-7)$ and $3\nmid (k-8)$. Since $k\not\equiv6~(\mbox{mod}~9)$, therefore the highest exponent of $3$ that can divide the numerator of the coefficient of $X^{(k-8)d}$ is $1$ and the highest exponent of $3$ that divides $4!$ is $1$, therefore the coefficient of $X^{(k-8)d}$ is nonzero. Since $\mathfrak D_{k,2}$ is DO polynomial, therefore $(k-2)d=3^{\beta}+1$ and $(k-8)d=3^{\gamma}+1$. Now since $k\geq15$, therefore $\beta, \gamma \geq 2$. Combining these two equations, we have $6d = 3^{\beta}-3^{\gamma}$ but then $3\mid d$ which is a contradiction as $(d,3)=1$. Thus, our assumption that $\mathfrak D_{k,2}$ is DO is wrong. This proves the desired result.

\textbf{Subcase 3.2} Let $p>3$. In this case, $k\not\equiv2,3,4~(\mbox{mod}~p)$. Therefore, the coefficients of $X^{(k-2)d}$ and $X^{(k-4)d}$ are nonzero. Thus, if $\mathfrak D_{k,2}$ is a DO polynomial then $(k-2)d=p^{\alpha}+1$ and $(k-4)d=p^{\beta}+1$. Combining these two equations, we have $2d=p^{\alpha}-p^{\beta}$. Since $k\geq7$, we must have $\alpha, \beta \geq1$, which implies that $p\mid d$. This leads to a contradiction. Therefore, $\mathfrak D_{k,2}$ is not DO in this case.
\end{proof}

\section{DO polynomials from the Polynomial $\mathfrak{D}_{k,3}$ }\label{S3}
Recall that $\mathfrak{D}_{k,3} = D_{k,3}(X^d,a)-D_{k,3}(0,a)$ which is given by
\begin{equation*}
 \mathfrak D_{k,3} =\sum_{i=0}^{\lfloor \frac{k-1}{2} \rfloor} \frac{k-3i}{k-i} {k-i \choose i} (-a)^i X^{(k-2i)d}
\end{equation*}
i.e.,
\begin{equation}\label{gk}
\begin{split}
 \mathfrak D_{k,3}=X^{kd}-(k-3)aX^{(k-2)d}+\frac{(k-3)(k-6)}{2!}a^2X^{(k-4)d}- \\
 \frac{(k-4)(k-5)(k-9)}{3!}a^3X^{(k-6)d}+\cdots.
\end{split}
\end{equation} 
Since, for $p=3$, $\mathfrak D_{k,3} = \mathfrak D_{k,0}$. Thus, for the case $p=3$, $\mathfrak D_{k,3}$ will be a DO polynomial whenever $\mathfrak D_{k,0}$ is DO. Consequently, the case $p=3$, follows immediately from \cite[Theorem 2.1]{CM}. For $p\geq5$, we prove the following theorem.  

\begin{thm} Let $q$ be a power of odd prime $p\geq5$. The polynomial $\mathfrak D_{k,3} $ is DO polynomial over $\mathbb{F}_q$ if and only if one of the following holds.

\begin{enumerate}[(i)]
   \item $k=1$ and $d= p^n(p^{\alpha}+1)$ for non-negative integers $\alpha$ and $n$.
   \item $k=2$ and $d= p^n(p^{\alpha}+1)/2$ for non-negative integers $\alpha$ and $n$.
   \item $k=3$ and $d= p^n(p^{\alpha}+1)/3$ for non-negative integers $\alpha$ and $n$.
   \item $k=5$, $p=5$ and $d= 2p^n$ for non-negative integer $n$.
\end{enumerate}
\end{thm}

\begin{proof}
It is enough to prove only the sufficient part of the theorem. Note that the polynomials $\mathfrak D_{1,3}=X^d$, $\mathfrak D_{2,3}=X^{2d}$ and $\mathfrak D_{3,3}=X^{3d}$ will be DO polynomial only if $d$ is of the form $p^{\alpha}+1$, $(p^{\alpha}+1)/2$ and $(p^{\alpha}+1)/3$, respectively. If the polynomial $\mathfrak D_{4,3}=X^{4d}-aX^{2d}$ is a DO polynomial then $2d = p^{\alpha}+1$ and $4d = p^{\beta}+1$. Hence, $2p^{\alpha}+1 = p^{\beta}$, which forces $\alpha = 0$ and $p^{\beta}=3$. This would lead to a contradiction as $p>3$. Therefore, $\mathfrak D_{4,3}$ is not DO. The polynomial $\mathfrak D_{5,3}=X^{5d}-2aX^{3d}-a^2X^{d}$ will be a DO polynomial only if $d = p^{\alpha}+1$, $3d = p^{\beta}+1$ and $5d = p^s(p^{\gamma}+1)$. Combining first two equations, we have $3p^{\alpha}+2 = p^{\beta}$, which is true if and only if $\alpha = 0$, $\beta=1$, $p=5$ and $d=2$. Now putting these values into third equation, we have $10=5^s(5^{\gamma}+1)$, which gives $s=1$ and $\gamma=0$. If the polynomial $\mathfrak D_{6,3}=X^{6d}-3aX^{4d}$ is  DO then $4d = p^{\alpha}+1$ and $6d = p^{\beta}+1$. Combining these two equations, we have $3p^{\alpha}+1 = 2p^{\beta}$ which forces $\alpha = 0$ and $p^{\beta}=2$. This gives a contradiction. Therefore, $\mathfrak D_{6,3}$ is not DO. The polynomial $\mathfrak D_{7,3}=X^{7d}-4aX^{5d}+2a^2X^{3d}+2a^3X^{d}$ will be a DO polynomial only if $d = p^{\alpha}+1$, $3d = p^{\beta}+1$, $5d = p^s(p^{\gamma}+1)$ and $7d = p^t(p^{\delta}+1)$. Combining first two equations, we have $3p^{\alpha}+2 = p^{\beta}$, which is true if and only if $\alpha = 0$, $\beta=1$, $d=2$ and $p=5$. Now putting this into fourth equation, we get  $t=0$ and $5^{\delta}=13$, which is a contradiction. Therefore, $\mathfrak D_{7,3}$ is not DO. For the polynomial $\mathfrak D_{8,3}=X^{8d}-5aX^{6d}+5a^2X^{4d}+2a^3X^{2d}$, we shall consider two cases, namely, $p=5$ and $p>5$. In the case $p=5$, coefficients of $X^{2d}$ and $X^{8d}$ are nonzero. Thus, if $\mathfrak D_{8,3}$ is a DO polynomial then $2d = 5^{\alpha}+1$ and $8d = 5^{\beta}+1$. Combining first two equations, we have $4\cdot 5^{\alpha}+3 = 5^{\beta}$, which forces $\alpha = 0$ and $5^{\beta}=7$. This leads to a contradiction. Therefore, $\mathfrak D_{8,3}$ is not a DO polynomial in this case. For the case $p>5$, the coefficients of $X^{2d}$, $X^{4d}$, $X^{6d}$ and $X^{8d}$ in $\mathfrak D_{8,3}$ are nonzero. Thus, $\mathfrak D_{8,3}$ will be a DO polynomial only if $2d = p^{\alpha}+1$, $4d = p^{\beta}+1$, $6d = p^{\gamma}+1$ and $8d = p^{\delta}+1$. Combining first two equations, we have $2p^{\alpha}+1 = p^{\beta}$, which forces $\alpha = 0$ and $p^{\beta}=3$. Thus, we arrive at a contradiction. Therefore, $\mathfrak D_{8,3}$ is not DO. For $k\geq 9$, we proceed by doing the following cases.

\textbf{Case 1} Let $k \not\equiv 0,2,3~(\mbox{mod}~p)$. In this case, the coefficients of $X^{kd}$ and $X^{(k-2)d}$ are nonzero. Therefore, if $\mathfrak D_{k,3}$ is a DO polynomial then $kd = p^{\alpha}+1$ and $(k-2)d = p^{\beta}+1$. Combining these two equations, we have $2d = p^{\alpha} - p^{\beta}$. Since $k \geq 9$, we must have $\alpha, \beta \geq 1$, which implies that $p \mid d$. This leads to a contradiction.

\textbf{Case 2 }Let $k \equiv 0~(\mbox{mod}~p)$. In this case, $k \not\equiv 2,3,4,6~(\mbox{mod}~p)$ and hence, the coefficients of $X^{(k-2)d}$ and $X^{(k-4)d}$ are nonzero. Thus, $\mathfrak D_{k,3}$ will be a DO polynomial only if $(k-2)d = p^{\alpha}+1$ and $(k-4)d = p^{\beta}+1$. Combining these two equations, we have $2d = p^{\alpha}-p^{\beta}$. Since $k \geq 9$, we have $\alpha, \beta \geq 1$, which implies that $p\mid d$. This gives a contradiction.

\textbf{Case 3} Let $k \equiv 2~(\mbox{mod}~p)$. In this case, $k \not\equiv 0,3,4,6~(\mbox{mod}~p)$ and hence, the coefficients of $X^{kd}$ and $X^{(k-4)d}$ are nonzero. Therefore, if $\mathfrak D_{k,3}$ is a DO polynomial then $kd = p^{\alpha}+1$ and $(k-4)d = p^{\beta}+1$. Combining these two equations, we have $4d = p^{\alpha}-p^{\beta}$. Since $k \geq 9$, we must have $\alpha, \beta \geq 1$, which implies that $p \mid d$. Thus, we arrive at a contradiction.

\textbf{Case 4} Let $k \equiv 3~(\mbox{mod}~p)$. In this case $k \not\equiv 0,4,5,6,9~(\mbox{mod}~p)$ and hence the coefficients of $X^{kd}$ and $X^{(k-6)d}$ are nonzero, therefore $\mathfrak D_{k,3}$ will be DO polynomial only if $kd = p^{\alpha}+1$ and $(k-6)d = p^{\beta}+1$. Combining these two equations, we have $6d = p^{\alpha}-p^{\beta}$. Since $k \geq 9$, therefore $\alpha, \beta \geq 1$ which implies that $p \mid d$, a contradiction.
\end{proof}

\section{DO polynomials from the polynomial $\mathfrak{D}_{k,4}$ }\label{S4}

Recall that $\mathfrak{D}_{k,4} = D_{k,4}(X^d,a)- D_{k,4}(0,a)$ whose expression is given by
\begin{equation*}
 \mathfrak D_{k,4} =\sum_{i=0}^{\lfloor \frac{k-1}{2} \rfloor} \frac{k-4i}{k-i} {k-i \choose i} (-a)^i X^{(k-2i)d}
\end{equation*}
i.e.,
\begin{equation}\label{hk}
 \mathfrak D_{k,4} =X^{kd}-(k-4)aX^{(k-2)d}+\frac{(k-3)(k-8)}{2!}a^2X^{(k-4)d}- \cdots.
\end{equation} 
Since, for $p=3$, $\mathfrak D_{k,4} = \mathfrak D_{k,1}$. Thus, for $p=3$, $\mathfrak D_{k,4}$ will be DO polynomial whenever $\mathfrak D_{k,1}$ is DO. As a consequence, the case $p=3$ follows immediately from \cite[Theorem 3.1]{CM}. For $p\geq5$, we prove the following theorem.

\begin{thm} Let $q$ be a power of odd prime $p\geq5$. The polynomial $\mathfrak D_{k,4}$ is DO polynomial over $\mathbb{F}_q$ if and only if one of the following holds.
\begin{enumerate}[(i)]
   \item $k=1$ and $d= p^n(p^{\alpha}+1)$ for non-negative integers $\alpha$ and $n$.
   \item $k=2$ and $d= p^n(p^{\alpha}+1)/2$ for non-negative integers $\alpha$ and $n$.
   \item $k=4$ and $d= p^n(p^{\alpha}+1)/4$ for non-negative integers $\alpha$ and $n$.
   \item $k=3,5$, $p=5$ and $d= 2p^n$ for non-negative integer $n$.
\end{enumerate}
\end{thm}

\begin{proof}
The necessary part of the theorem is obvious. Thus, we shall focus only on the sufficient part. It may be noted that $\mathfrak D_{1,4}=X^{d}$, $\mathfrak D_{2,4}=X^{2d}$ and $\mathfrak D_{4,4}=X^{4d}$ will be a DO polynomial only if $d$ is of the form $p^{\alpha}+1$, $(p^{\alpha}+1)/2$ and $(p^{\alpha}+1)/4$, respectively. If the polynomial $\mathfrak D_{3,4}=X^{3d}+aX^{d}$ is a DO polynomial then $d = p^{\alpha}+1$ and $3d = p^{\beta}+1$. Combining these two equations, we have $3p^{\alpha}+2 = p^{\beta}$, which is true if and only if $\alpha = 0$, $\beta=1$, $p=5$ and $d=2$. Again, if the polynomial $\mathfrak D_{5,4}=X^{5d}-aX^{3d}-3a^2X^{d}$ is a DO polynomial then $d = p^{\alpha}+1$ , $3d = p^{\beta}+1$ and $5d = p^s(p^{\gamma}+1)$. Combining first two equations, we have $3p^{\alpha}+2 = p^{\beta}$, which is true if and only if $\alpha = 0$, $\beta=1$, $p=5$ and $d=2$. Now putting these values into third equation, we have $10=5^s(5^{\gamma}+1)$, which gives $s=1$ and $\gamma=0$. The polynomial $\mathfrak D_{6,4}=X^{6d}-2aX^{4d}-3a^2X^{2d}$ will be a DO polynomial only if $2d = p^{\alpha}+1$, $4d = p^{\beta}+1$ and $6d = p^{\gamma}+1$. Combining first two equations, we get $2p^{\alpha}+1 = p^{\beta}$, which forces $\alpha = 0$ and $p^{\beta}=3$ and this leads to a contradiction. Therefore, $\mathfrak D_{6,4}$ is not DO. Further, if the polynomial $\mathfrak D_{7,4}=X^{7d}-3aX^{5d}-2a^2X^{3d}+5a^3X^{d}$ is a DO polynomial then $d = p^{\alpha}+1$, $3d = p^{\beta}+1$, $5d = p^s(p^{\gamma}+1)$ and $7d = p^t(p^{\delta}+1)$. Combining first two equations, we have $3p^{\alpha}+2 = p^{\beta}$ which forces $\alpha = 0$, $\beta=1$, $d=2$ and $p=5$. Now putting these values into fourth equation, we have $t=0$ and $5^{\delta}=13$ which is a contradiction. Again, if the polynomial $\mathfrak D_{8,4}=X^{8d}-4aX^{6d}+8a^3X^{2d}$ is a DO polynomial then $2d = p^{\alpha}+1$, $6d=p^{\beta}+1$ and $8d = p^{\gamma}+1$. Combining first two equations, we have $3p^{\alpha}+2 = p^{\beta}$ which is true if and only if $\alpha = 0$, $\beta=1$, $d=1$ and $p=5$. Now putting these values into third equation we have $5^{\gamma}=7$, which is a contradiction. For the polynomial $\mathfrak D_{9,4}=X^{9d}-5aX^{7d}+3a^2X^{5d}+10a^3X^{3d}-7a^4X^{d}$, we shall consider two cases, namely, $p=5$ and $p>5$. For the case $p=5$, the coefficients of $X^{d}$, $X^{5d}$ and $X^{9d}$ in $\mathfrak D_{9,4}$ are nonzero. Thus, if $\mathfrak D_{9,4}$ is a DO polynomial then $d = 5^{\alpha}+1$, $5d = 5^s(5^{\beta}+1)$ and $9d = 5^{\gamma}+1$. Since $5^s \mid 5$, we have $s=1$ and hence the second equation reduces to $d=5^{\beta}+1$. Combining first and third equations, we have $9p^{\alpha}+8 = p^{\beta}$, which forces $\alpha = 0$, $p^{\gamma}=17$ and this gives a contradiction. For the case $p>5$, the coefficients of $X^{3d}$, $X^{5d}$ and $X^{9d}$ are nonzero. Thus, $\mathfrak D_{9,4}$ will be a DO polynomial only if $3d = p^{\alpha}+1$, $5d = p^{\beta}+1$, and $9d = p^{\gamma}+1$. Combining first two equations, we have $5p^{\alpha}+2 = 3p^{\beta}$, which forces $\alpha = 0$ and $3p^{\beta}=7$. This gives a contradiction. Similarly, for the polynomial $\mathfrak D_{10,4}=X^{10d}-6aX^{8d}+7a^2X^{6d}+10a^3X^{4d}-15a^4X^{2d}$, we shall consider two cases, namely, $p=5$ and $p>5$. For the case $p=5$, the coefficients of $X^{6d}$, $X^{8d}$ and $X^{10d}$ are nonzero. Thus if $\mathfrak D_{10,4}$ is a DO polynomial then $6d = 5^{\alpha}+1$, $8d=5^{\beta}+1$ and $10d=5^s(5^{\gamma}+1)$. Combining first two equations, we have $4\cdot5^{\alpha}+1 = 3\cdot5^{\beta}$, which forces $\alpha = 0$ and $3\cdot5^{\beta}=5$ and this leads to a contradiction. If $p>5$ then the coefficients of $X^{2d}$ and $X^{4d}$ are nonzero. Therefore, $\mathfrak D_{10,4}$ will be a DO polynomial only if $2d = p^{\alpha}+1$ and $4d=p^{\beta}+1$. Combining these two equations, we have $2p^{\alpha}+1 = p^{\beta}$, which forces $\alpha = 0$ and $p^{\beta}=3$. This gives a contradiction and therefore, $\mathfrak D_{10,4}$ is not DO. For $k\geq 11$, we proceed by considering various cases.

\textbf{Case 1} Let $k \not\equiv 0,2,4~(\mbox{mod}~p)$. In this case, the coefficients of $X^{kd}$ and $X^{(k-2)d}$ in $\mathfrak D_{k,4}$ are nonzero. Therefore, if $\mathfrak D_{k,4}$ is a DO polynomial then $kd = p^{\alpha}+1$ and $(k-2)d = p^{\beta}+1$. Combining these two equations, we have $2d = p^{\alpha} - p^{\beta}$. Since $k \geq 11$, we have $\alpha, \beta \geq 1$, which implies that $p \mid d$. This leads to a contradiction. Therefore, $\mathfrak D_{k,4}$ is not DO.

\textbf{Case 2} Let $k \equiv 0~(\mbox{mod}~p)$. In this case, $k \not\equiv 2,3,4,8~(\mbox{mod}~p)$ and hence, the coefficients of $X^{(k-2)d}$ and $X^{(k-4)d}$ in $\mathfrak D_{k,4}$ are nonzero. Therefore, $\mathfrak D_{k,4}$ will be DO polynomial only if $(k-2)d = p^{\alpha}+1$ and $(k-4)d = p^{\beta}+1$. Combining these two equations, we have $2d = p^{\alpha}-p^{\beta}$. Since $k \geq 11$, we must have $\alpha, \beta \geq 1$, which implies that $p \mid d$. This gives a contradiction, therefore, $\mathfrak D_{k,4}$ is not DO.

\textbf{Case 3} Let $k \equiv 2~(\mbox{mod}~p)$. In this case, $k \not\equiv 0,3,4,8~(\mbox{mod}~p)$ and hence, the coefficients of $X^{kd}$ and $X^{(k-4)d}$ in $\mathfrak D_{k,4}$ are nonzero. Therefore, if $\mathfrak D_{k,4}$ is a DO polynomial then $kd = p^{\alpha}+1$ and $(k-4)d = p^{\beta}+1$. Combining these two equations, we have $4d = p^{\alpha}-p^{\beta}$. Since $k \geq 11$, we have $\alpha, \beta \geq 1$ which implies that $p \mid d$. This leads to a contradiction, therefore, $\mathfrak D_{k,4}$ is not DO.

\textbf{Case 4} Let $k \equiv 4~(\mbox{mod}~p)$. Consider the fifth term in $\mathfrak D_{k,4}$ given by $$ \frac{(k-5)(k-6)(k-7)(k-16)}{4!}a^4X^{(k-8)d}. $$ Since, $k \not\equiv 0,5,6,7,8,16~(\mbox{mod}~p)$ and hence, the coefficients of $X^{kd}$ and $X^{(k-8)d}$ in $\mathfrak D_{k,4}$ are nonzero. Therefore, if $\mathfrak D_{k,4}$ is a DO polynomial, then $kd = p^{\alpha}+1$ and $(k-8)d = p^{\beta}+1$. Combining these two equations, we have $8d = p^{\alpha}-p^{\beta}$. Since $k \geq 11$, we have $\alpha, \beta \geq 1$ which implies that $p \mid d$. Which is a contradiction, hence $\mathfrak D_{k,4}$ is not DO.
\end{proof}

\section{The case $m \geq 5$ }\label{S5}
As alluded to in Introduction, we denote the polynomial $D_{k,m}(X^d,a)-D_{k,m}(0,a)$ by $\mathfrak D_{k,m}$, whose expression is given by
\begin{equation*}
 \mathfrak D_{k,m} =\sum_{i=0}^{\lfloor \frac{k-1}{2} \rfloor} \frac{k-mi}{k-i} {k-i \choose i} (-a)^i X^{(k-2i)d}
\end{equation*}
i.e.,
\begin{equation}\label{dkm}
 \mathfrak D_{k,m}=X^{kd}-(k-m)aX^{(k-2)d}+\frac{(k-3)(k-2m)}{2!}a^2X^{(k-4)d}-\cdots.
\end{equation}

In this section, we find conditions on $k$, $m$, $p$ and $d$ for which the polynomial $\mathfrak D_{k,m}$ is DO polynomial. For the case $p=3$, we have either $m\equiv0 ~(\mbox{mod}~3)$ or $m\equiv1 ~(\mbox{mod}~3)$ or $m\equiv2~(\mbox{mod}~3)$. Therefore, in these cases, the polynomial $\mathfrak D_{k,m}$ will be a DO polynomial whenever $\mathfrak D_{k,0}$, $\mathfrak D_{k,1}$ and $\mathfrak D_{k,2}$ are DO, respectively. Thus, for the case $p=3$, the classification follows from \cite[Theorem 2.1 and 3.1]{CM} and from the Section \ref{S2} of this paper, respectively. Similarly, for $p=5$, we have either $m\equiv0~(\mbox{mod}~5)$ or $m\equiv1~(\mbox{mod}~5)$ or $m\equiv2~(\mbox{mod}~5)$ or $m\equiv3~(\mbox{mod}~5)$ or $m\equiv4~(\mbox{mod}~5)$. In these cases, the polynomial $\mathfrak D_{k,m}$ will be a DO polynomial whenever $\mathfrak D_{k,0}$, $\mathfrak D_{k,1}$, $\mathfrak D_{k,2}$, $\mathfrak D_{k,3}$ and $\mathfrak D_{k,4}$ are DO, respectively. In the case $p>5$, the cases $m\equiv 0,1,2,3,4~(\mbox{mod}~p)$ are already classified in \cite[Theorem 2.1 and 3.1]{CM} and in Section \ref{S2}, \ref{S3} and \ref{S4} of this paper, respectively. For the case $p>5$ and $m\not \equiv 0,1,2,3,4 ~(\mbox{mod}~p)$, we prove the following theorem.

\begin{thm} Let $q$ be a power of odd prime $p>5$ and 
$m\not\equiv0,1,2,3,4 \pmod p$. The polynomial 
$\mathfrak D_{k,m}$ is DO polynomial over $\mathbb{F}_q$ if and only if one of the following holds.
\begin{enumerate}[(i)]
   \item $k=1$ and $d= p^n(p^{\alpha}+1)$ for non-negative integers $\alpha$ and $n$.
   \item $k=2$ and $d= p^n(p^{\alpha}+1)/2$ for non-negative integers $\alpha$ and $n$.
\end{enumerate}
\end{thm}

\begin{proof} 

Since necessary part of the theorem is clear, we shall prove only the sufficient part. Notice that $\mathfrak D_{1,m}=X^d$ and $\mathfrak D_{2,m}=X^{2d}$ will be  DO polynomial only if $d$ is of the form $p^{\alpha}+1$ and $(p^{\alpha}+1)/2$, respectively. For the polynomial $\mathfrak D_{3,m}=X^{3d}+(m-3)aX^d$, the coefficient of $X^{d}$ would be nonzero. Otherwise, $m\equiv 3 \pmod p$, which will lead to a contradiction to fact that $m\not\equiv 3 \pmod p$. Therefore, if $\mathfrak D_{3,m}$ is a DO  polynomial then $d = p^{\alpha}+1$ and $3d = p^{\beta}+1$. Combining these two equations, we have $3p^{\alpha}+2 = p^{\beta}$, which forces $\alpha = 0$ and $p^{\beta}=5$. This gives a contradiction and therefore, $\mathfrak D_{3,m}$ is not DO. For the polynomial $\mathfrak D_{4,m}=X^{4d}+(m-4)aX^{2d}$, the coefficient of $X^{2d}$ is nonzero. Therefore, $\mathfrak D_{4,m}$ will be DO only if $2d = p^{\alpha}+1$ and $4d = p^{\beta}+1$. Combining these two equations, we have $2p^{\alpha}+1 = p^{\beta}$, which forces $\alpha = 0$ and $p^{\beta}=3$. This leads to a contradiction. Therefore, $\mathfrak D_{4,m}$ is not DO. Note that in the polynomial $\mathfrak D_{5,m}=X^{5d}+(m-5)aX^{3d}+(5-2m)a^2X^{d}$, the coefficients of  $X^d$ and $X^{3d}$ can not be simultaneously zero. Otherwise, $2m\equiv 5 \pmod p$ and $m\equiv 5 \pmod p$ would imply that $m\equiv0 \pmod p$, which is a contradiction as $m\not\equiv 0 \pmod p$. In the case, when $m\equiv 5 \pmod p$, if $\mathfrak D_{5,m}$ is a DO polynomial then $d = p^{\alpha}+1$ and $5d = p^{\beta}+1$. Combining these two equations, we have $5p^{\alpha}+4 = p^{\beta}$, which forces $\alpha = 0$ and $p^{\beta}=9$. This gives a contradiction. Therefore, $\mathfrak D_{5,m}$ is not DO. Further, in the case, when $m \not\equiv 5~(\mbox{mod}~p)$, if $\mathfrak D_{5,m}$ is a DO polynomial then $3d = p^{\alpha}+1$ and $5d = p^{\beta}+1$. Combining these two equations, we obtain $5p^{\alpha}+2 = 3p^{\beta}$, which forces $\alpha = 0$ and $3p^{\beta}=7$. This leads to a contradiction and therefore, $\mathfrak D_{5,m}$ is not DO. For the polynomial $\mathfrak D_{6,m}= X^{6d}+(m-6)aX^{4d}+3(3-m)a^2X^{2d}$, the coefficient of  $X^{2d}$ would be nonzero. Otherwise, $m\equiv3 \pmod p$, which is a contradiction. Therefore, if $\mathfrak D_{6,m}$ is a DO polynomial, then $2d = p^{\alpha}+1$ and $6d = p^{\beta}+1$. Combining these two equations, we get $3p^{\alpha}+2 = p^{\beta}$, which forces $\alpha = 0$ and $p^{\beta}=5$ and we arrive at a contradiction. Therefore, $\mathfrak D_{6,m}$ is not DO. Now for $k \geq 7$, we have following cases.

\textbf{Case 1} Let $k \not\equiv 0,2,m~(\mbox{mod}~p)$. In this case, the coefficients of $X^{kd}$ and $X^{(k-2)d}$ are nonzero. Therefore, if $\mathfrak D_{k,m}$ is a DO polynomial, then $kd = p^{\alpha}+1$ and $(k-2)d = p^{\beta}+1$. Combining these two equations, we obtain $2d =p^{\alpha}-p^{\beta}$. Note that we must have for $k\geq7$, $\alpha, \beta \geq 1$, which forces $p \mid d$. Thus, we arrive at a contradiction. Therefore, $\mathfrak D_{k,m}$ is not DO.

\textbf{Case 2} Let $k\equiv0~(\mbox{mod}~p)$. In this case, $k\not\equiv 2,3,4~(\mbox{mod}~p)$. Also, we observe that $k \not\equiv m~(\mbox{mod}~p)$, otherwise, $m\equiv0~(\mbox{mod}~p)$, which is a contradiction. Similarly, $k \not\equiv 2m~(\mbox{mod}~p)$, otherwise $m\equiv0~(\mbox{mod}~p)$, which is again a contradiction. Therefore, in this case, the coefficients of $X^{(k-2)d}$ and $X^{(k-4)d}$ are nonzero. Thus, if $\mathfrak D_{k,m}$ is a DO polynomial, then $(k-2)d = p^{\alpha}+1$ and $(k-4)d = p^{\beta}+1$. Combining these two equations, we have $2d =p^{\alpha}-p^{\beta}$. For $k\geq7$, we must have $\alpha, \beta \geq 1$, which forces $p \mid d$. This leads to a contradiction. Therefore, $\mathfrak D_{k,m}$ is not DO.

\textbf{Case 3} Let $k\equiv2~(\mbox{mod}~p)$. In this case, $k\not\equiv 0,3~(\mbox{mod}~p)$. Also, we observe that $k \not\equiv 2m~(\mbox{mod}~p)$. Otherwise $m\equiv1~(\mbox{mod}~p)$, which would lead to a contradiction. In this case, the coefficients of $X^{kd}$ and $X^{(k-4)d}$ in $\mathfrak D_{k,m}$ are nonzero. Therefore, $\mathfrak D_{k,m}$ will be a DO polynomial only if $kd = p^{\alpha}+1$ and $(k-4)d = p^{\beta}+1$. Combining these two equations, we have $4d =p^{\alpha}-p^{\beta}$. Now for $k\geq7$, we must have $\alpha, \beta \geq 1$, which forces $p \mid d$. Thus, we arrive at a contradiction, Hence, $\mathfrak D_{k,m}$ is not DO.

\textbf{Case 4} Let $k \equiv m~(\mbox{mod}~p)$. In this case, $k \not\equiv 0,3,4~(\mbox{mod}~p)$. Also, we observe that $k \not\equiv 2m~(\mbox{mod}~p)$. Otherwise $m\equiv0~(\mbox{mod}~p)$, which would lead to a contradiction. In this case, the coefficients of $X^{kd}$ and $X^{(k-4)d}$ in $\mathfrak D_{k,m}$ are nonzero. Therefore, if $\mathfrak D_{k,m}$ is a DO polynomial, then $kd = p^{\alpha}+1$ and $(k-4)d = p^{\beta}+1$. Combining these two equations, we have $4d =p^{\alpha}-p^{\beta}$. For $k\geq7$, we have $\alpha, \beta \geq 1$, which forces $p \mid d$. This gives a contradiction. Therefore, $\mathfrak D_{k,m}$ is not DO.
\end{proof}

\section{Discussion on planarity}\label{S6}

We shall first discuss some tools and techniques, also described in \cite{FHP}, required to handle planarity of the DO polynomials obtained in the previous sections and listed in Appendix~\ref{A1}. As defined earlier, a function $f$ is planar if the difference function $\Delta_f(X, \epsilon) =f(X+\epsilon)-f(X)-f(\epsilon)$ is a permutation of $\f_{p^e}$ for all $\epsilon \in \mathbb{F}_{p^e}^*$. When $f$ is a DO polynomial (which is what we shall always assume in our discussion below), $\Delta_f (X, \epsilon)$ belongs to another well-known class of polynomials called linearized polynomial, for each $\epsilon \in \mathbb{F}_{p^e}^*$. Therefore $f$ is planar if and only if linearized polynomial $\Delta_f (X,\epsilon)$ is permutation polynomial for all $\epsilon \in \mathbb{F}_{p^e}^*$. The permutation behaviour of linearized polynomial is well-known. In fact, \cite[Theorem 7.9]{BOOK} tells us that a linearized polynomial is a permutation polynomial over $\mathbb{F}_{q}$ if and only if its only root in $\mathbb{F}_q$ is $0$. Thus, in order to show that $f$ is not planar, it is sufficient to exhibit that the difference function $\Delta_{f}(X,Y)= f(X+Y)-f(X)-f(Y)$ has a non-zero root $(X,Y)\in \mathbb{F}_{p^e}^* \times \mathbb{F}_{p^e}^*$. 

We recall that a function $f$ from $\mathbb{F}_q$ to itself is two-to-one (2-to-1) if for all but one $b \in \mathbb{F}_q$, it has either 2 or 0 preimages of $f$, and the exception element has exactly one preimage. For a DO polynomial $f$, Weng and Zeng  \cite{planar} gave the following necessary and sufficient condition for $f$ to be a planar function.
\begin{lem}\cite[Theorem 2.3]{planar} \label{l6.1}
Let $f$ be a DO polynomial over $\mathbb{F}_q$. Then $f$ is planar if and only if $f$ is 2-to-1.
\end{lem}
Notice that DO polynomials are even function. Therefore from Lemma~\ref{l6.1}, it is straightforward to see that if a DO polynomial $f$ has a root $z\in \mathbb{F}_q^*$ then $f$ is not planar. Also, it is easy to see that if $f(X)$ is a DO polynomial then so is $f(X^{p^n})$. Since $X^{p^n}$ is a linearized permutation polynomial over $\mathbb{F}_{p^e}$, the cardinality of the image set of $f(X)$ and  $f(X^{p^n})$ on $\mathbb{F}_q^*$ is same. Hence if $f(X)$ is planar then $f(X^{p^n})$ is also planar. Thus in such a case, it would be sufficient to consider the planarity of $f(X)$. 

We shall use the following  version of Weil bound to calculate the number of roots of certain absolutely irreducible bivariate polynomials over finite fields. 
\begin{lem} \cite[Lemma 2.4]{JS} \label{WB}
Let $f(X,Y)$ be an absolutely irreducible polynomial in $\mathbb{F}_q[X, Y]$ of degree $d$ and
let $N_f$ be the number of zeros of $f$. Then 
$$q+1-(d-1)(d-2)\sqrt{q}-d \leq N_f \leq q+1+(d-1)(d-2)\sqrt{q}.$$
\end{lem}
The strategy of using the above lemma is as follows. For any DO polynomial $f$, we shall calculate the bivariate polynomial $\Delta_f(X,Y)$. Using the computer algebra package Magma, we shall check the absolute irreducibility of $\Delta_f(X,Y)$. If $\Delta_f(X,Y)$ is absolutely irreducible then the Lemma~\ref{WB} will give a bound on the number of solution $N_{\Delta_f}$ of $\Delta_f(X,Y)$ in $\mathbb{F}_{p^e} \times \mathbb{F}_{p^e}$. We shall then calculate the number of solutions of $\Delta_f(X,Y)$ which can be obtained either by putting $X=0$ or $Y=0$. If $N_{\Delta_f}$ is strictly greater than the number of solutions corresponding to $XY=0$, there will exist a solution $(u,v) \in \mathbb{F}_{p^e}^* \times \mathbb{F}_{p^e}^*$ of $\Delta_f(X,Y)$ and as a consequence, $f$ will not be a planar function.

We shall consider the planarity of the DO polynomials listed in the Appendix \ref{A1} in three different cases, namely, $p=3$, $p=5$ and $p>5$, respectively.

\textbf{Case 1}~Let $p=3$. The planarity of the DO polynomials arising from the polynomials $\mathfrak{D}_{k,0}$ and $\mathfrak{D}_{k,1}$ has already been discussed \cite{CM}. Therefore, it only remains to consider the planarity of the DO polynomials obtained from the polynomial $\mathfrak{D}_{k,2}$. 

(\romannumeral 1)
The planarity of the monomial $X^{3^{\alpha}+1}$ is well-known and it is planar over $\mathbb{F}_{3^e}$ if and only if $ \frac{e}{\gcd(\alpha,e)}$ is odd.

(\romannumeral 2)
The DO polynomial $X^{3(3^{\alpha}+1)}+2aX^{3^{\alpha}+1}$ is planar if and only if $X^3+2aX$ is a permutation polynomial and $X^{3^{\alpha}+1}$ is planar. Therefore it is planar over $\mathbb{F}_{3^e}$ if and only if $ \frac{e}{\gcd(\alpha,e)}$ is odd and $a$ is not a square in $\mathbb{F}_{3^e}$.

(\romannumeral 3) 
In the case of DO polynomial $X^4+aX^2$, consider the difference function 
$$\Delta_f(X,Y)= XYh(X,Y),$$
where $h(X,Y) =X^2+Y^2-a.$ Now, let $z$ be a solution of the equation $X^2-a$ in the algebraic closure of $\mathbb{F}_{3^e}$, then $Y-z \mid Y^2-a$ and since $Y^2-a$ has no repeated root, by Eisenstein's criterion, $X^2+Y^2-a$ is absolutely irreducible. Hence by Lemma \ref{WB}, the number of solution $N_h$ of $h(X,Y)$ in $\mathbb{F}_{3^e}\times\mathbb{F}_{3^e}$ is greater than or equal to $3^e-1$. Now, at most 4 solutions of $h(X,Y)$ can be obtained either by putting $X=0$ or $Y=0$. Hence, when $3^e-1>4$, there must be a solution $(x,y)\in \mathbb{F}_{3^e}^* \times \mathbb{F}_{3^e}^*$. This holds for $e\geq2$, thus $X^4+aX^2$ is not planar over $\mathbb{F}_{3^e}$ for $e\geq2$. For $e=1$, $X^4+aX^2 \equiv (1+a)X^2 \pmod {X^3-X}$. Now since $a\in \mathbb{F}_3^*$  therefore the only choice for $a$ is $1$ or $2$. When $a=1$, $X^2$ is clearly a planar function. When $a=2$, it is a zero polynomial and hence can not be planar.

(\romannumeral 4)
For the DO polynomial $f(X)=X^{10}+a^2X^2$, consider the difference function 
$$\Delta_f(X,Y)= XYh(X,Y),$$
where $h(X,Y) =X^8+Y^8-a^2.$ Now, let $z$ be a solution of the equation $X^8-a^2$ in the algebraic closure of $\mathbb{F}_{3^e}$ then $Y-z \mid Y^8-a^2$ and since $Y^8-a^2$ has no repeated root, by Eisenstein's criterion, $X^8+Y^8-a^2$ is absolutely irreducible. Hence by Lemma \ref{WB}, the number of solution $N_h$ of $h(X,Y)$ in $\mathbb{F}_{3^e}\times\mathbb{F}_{3^e}$ is greater than or equal to $3^e-42\sqrt{3^e}-7$. Now, at most $16$ solutions of $h(X,Y)$ can be obtained either by putting $X=0$ or $Y=0$. Hence when $3^e-42\sqrt{3^e}-7>16$, there must be a solution $(x,y)\in \mathbb{F}_{3^e}^* \times \mathbb{F}_{3^e}^*$. This holds for $e\geq7$, thus $f$ is not planar over $\mathbb{F}_{3^e}$ for $e\geq7$. For $e=1$, $f =(1+a^2)X^2$. Now since $1+a^2 \neq 0$ for all $a\in \mathbb{F}_3^*$, $f$ is clearly a planar function for all $a\in \mathbb{F}_3^*$. Similarly, for $e=2$, $f= (1+a^2)X^2 $, which is planar over $\mathbb{F}_{3^2}$ provided $1+a^2 \neq 0$. Computations show that for $e=3,5$ and $6$, $f$ is not a planar function over $\mathbb{F}_{3^e}$ for all $a\in \mathbb{F}_{3^e}^*$. When $e=4$, 
computations show that $f$ is planar if and only if $a = g^{4n+2}$, where $g$ is the primitive element of $\mathbb{F}_{3^4}$.

(\romannumeral 5)
For the DO polynomial $f(X)=X^{6}+2aX^4$, consider the difference function 
$$\Delta_f(X,Y)= XYh(X,Y),$$
where $h(X,Y) =a(X^2+Y^2)+X^2Y^2$. The computation using Magma reveals that $h(X,Y)$ is absolutely irreducible. Now, at most one solution of $h(X,Y)$ can be obtained either by putting $X=0$ or $Y=0$. Hence, when $3^e-6\sqrt{3^e}-3>1$, there must be a solution $(x,y)\in \mathbb{F}_{3^e}^* \times \mathbb{F}_{3^e}^*$. This holds for $e\geq4$, thus  for $e\geq4$, $f$ is not planar over $\mathbb{F}_{3^e}$. For $e=1$, $f =(1-a)X^2$. Now since $a\in \mathbb{F}_3^*$, therefore the only choice for $a$ is $1$ or $2$. When $a=2$, $2X^2$ is clearly a planar function. When $a=1$, it is a zero polynomial and hence not planar. For $e=2$ and $3$, direct computations show that $f$ is not planar over $\mathbb{F}_{3^e}$ for all $a\in \mathbb{F}_{3^e}^*$.

(\romannumeral 6)
In the case of DO polynomial $f(X)=X^{12}+2aX^{10}+a^3X^6+2a^4X^4$, consider the difference function 
$$\Delta_f(X,Y)= XYh(X,Y),$$
where $h(X,Y) =X^8Y^2+X^2Y^8+2a(X^8+Y^8)+2a^3X^2Y^2+2a^4(X^2+Y^2)$. The computation using Magma reveals that $h(X,Y)$ is absolutely irreducible.
Now, at most $16$ solutions of $h(X,Y)$ can be obtained either by putting $X=0$ or $Y=0$. Hence, when $3^e-72\sqrt{3^e}-9>16$, there must be a solution $(x,y)\in \mathbb{F}_{3^e}^* \times \mathbb{F}_{3^e}^*$. This holds for $e\geq8$, thus  for $e\geq8$, $f$ is not planar over $\mathbb{ F}_{3^e}$. For $e=1$, $f(X)=(1-a)X^2$. Since $a\in \mathbb{F}_3^*$, therefore the only choice for $a$ is $1$ or $2$. When $a=2$, $2X^2$ is clearly a planar function. When $a=1$, it is a zero polynomial and hence can not be planar. The direct computations show that for $e=2$, $f$ is planar over $\mathbb{F}_{3^e}$ for all non-square $a\in \mathbb{F}_{3^2}^*$ and for $3\leq e \leq 7$, $f$ is not planar for all $a\in \mathbb{F}_{3^e}^*.$

(\romannumeral 7)
For the DO polynomial $f(X)=X^{36}+2aX^{28}+2a^3X^{12}+a^4X^4$, notice that $f(X) = (X^9+2aX^7+2a^3X^3+a^4X) \circ X^4$. Therefore, whenever $4\mid (3^e-1)$, i.e, $e$ is even, the cardinality of the image set of $f$ on $\mathbb{F}_{3^e}^*$ will be at most $(3^e-1)/4$ and hence by Lemma \ref{l6.1}, $f$ will not be planar in this case. When $e$ is odd, we could not adopt the strategy that we have used for previous polynomials as in this case the difference function $\Delta_f(X,Y)$ does not turn out to be absolutely irreducible. Therefore, in this case, we are able to provide only computational results for all odd $e$ such that $e<10$. For $e=1$, $f= (2+a)X^2 $, which is planar over $\mathbb{F}_3$ if and only if $a=2$. For $3\leq e \leq 9$ and $e$ odd , the direct computations show that $f$ is planar over $\mathbb{F}_{3^e}$ if $a$ is a non-square in $\mathbb{F}_{3^e}^*$.

\textbf{Case 2} Let $p=5$. The planarity of the DO polynomials obtained from the polynomials $\mathfrak{D}_{k,0}$ and $\mathfrak{D}_{k,1}$ has already been discussed in \cite{CM}. Therefore, we shall discuss the planarity of the DO polynomials obtained from the polynomials $\mathfrak{D}_{k,2}$, $\mathfrak{D}_{k,3}$ and $\mathfrak{D}_{k,4}$.

(\romannumeral 1)
The planarity of the DO polynomial $X^{5^{\alpha}+1}$ is well-known and it is planar over $\mathbb{F}_{5^e}$ if and only if $\frac{\alpha}{\gcd(\alpha,e)}$ is odd.

(\romannumeral 2)
The DO polynomials $X^6+4aX^2$ and $X^6+aX^2$ practically yield the same DO polynomial. Therefore we shall consider the planarity of $f(X)=X^6+4aX^2$. Now consider the difference function
$$\Delta_f(X,Y)=XYh(X,Y),$$
where $h(X,Y)=X^4+Y^4+3a$. Let $z$ be a root of the equation $Y^4+3a=0$ in the algebraic closure of $\mathbb{F}_q$. Then $Y-z\mid Y^4+3a$ and since $Y^4+3a$ has no repeated root, by Eisenstein's criterion, $h(X,Y)$ is absolutely irreducible. Therefore by Lemma \ref{WB} , $h(X,Y)$ will have atleast $5^e-6\sqrt{5^e}-3$ solutions in $\mathbb{F}_{5^e} \times \mathbb{F}_{5^e}$. Now at most $8$ roots can be obtained either by putting $X=0$ or $Y=0$. Therefore, if $5^e-6\sqrt{5^e}-11>0$ then $h(X,Y)$ will have a root in $\mathbb{F}_{5^e}^*\times\mathbb{F}_{5^e}^*$ and $f$ will not be planar. This holds for $e\geq3$ and hence for $e\geq3$, $f$ is not planar over $\mathbb{F}_{5^e}$. For $e=1$, $f(X) = (1-a)X^2$ and is planar over $\mathbb{F}_5$ provided $a\neq1$. In the case of $e=2$, $f$ is planar if and only if $a$ is of the form $g^{4i+3}$ where $g\in \mathbb{F}_{25}^*$ is a root of primitive polynomial $X^2+4X+2$ over $\mathbb{F}_5$.

(\romannumeral 3)
Similarly, the DO polynomials $X^{10}+2aX^6+a^2X^2$, $X^{10}+3aX^6+4a^2X^2$ and $X^{10}+4aX^6+2a^2X^2$ are practically same. Therefore we shall consider the planarity of $f(X)=X^{10}+2aX^6+a^2X^2 = X^2(X^4+a)^2$. Notice that when $-a$ is a fourth power in $\mathbb{F}_{5^e}$, then $f$ has a nonzero root and hence is not planar over $\mathbb{F}_{5^e}$. Therefore we shall consider the case when $-a$ is not a fourth power in $\mathbb{F}_{5^e}.$ For $e=1$, $f = (1+2a+a^2)X^2$ which is planar over $\mathbb{F}_{5}$ if and only if $a \neq 4$. For $e\geq 2$, consider the difference function
$$\Delta_f(X,Y)=2XYh(X,Y),$$
where $h(X,Y)=(a+Y^4)(X^4+a)$. Since $-a$ is not a fourth power in $\mathbb{F}_{5^e}$, thus $\Delta_f(X,Y)$ will not have any root in $\mathbb{F}_{5^e}^*\times \mathbb{F}_{5^e}^*.$ Therefore, for all $Y = \beta \in \mathbb{F}_{5^e}^*$, the linearized polynomial $(\beta^5+a\beta)(X^5+aX)$ is permutation polynomial and hence for $e\geq 2$, $f$ is planar for all $a \in \mathbb{F}_{5^e}$ such that $-a$ is not a fourth power in $\mathbb{F}_{5^e}^*.$

\textbf{Case 3} Let $p>5$. The only DO polynoimals obtained from the polynomial $\mathfrak{D}_{k,m}$ are monomials of the form $AX^{p^{\alpha}+1}$ with $A\in \mathbb{F}_{p^e}^*$, which is planar over $\mathbb{F}_{p^e}$ if and only if $\frac{\alpha}{\gcd(\alpha, e)}$ is odd.

In view of the foregoing discussions, we obtain the following list of planar DO polynomials obtained from the polynomials $\mathfrak D_{k,m}$ over $\mathbb{F}_{p^e}$.
\begin{thm}
 Let $\mathfrak{D}_{k,m}$ be a polynomial defined in the Introduction. Then for $k\geq 2$, the following are the only planar DO polynomials arising from $\mathfrak{D}_{k,m}$.
\begin{enumerate}
\item $X^2$ over $\mathbb{F}_{p^e}$.
 \item $X^{p^\alpha+1}$ over $\mathbb{F}_{p^e}$ with $e/(e,\alpha)$ odd.
 \item $X^{10}+aX^2$ over $\mathbb{F}_{3^4}$ and $a=g^{4n+2}$, where $g$ is the root of the primitive polynomial $X^4+2X^3+2$ over $\mathbb{F}_3.$
 \item $X^{12}+2aX^{10}+a^3X^6+2a^4X^4$ over $\mathbb{F}_{3^2}$, $a\in \mathbb{F}_{3^2}^*$ is not a square.
 \item $X^{36}+2aX^{28}+2a^3X^{12}+a^4X^4$ over $\mathbb{F}_{3^e}$, $1\leq e \leq 9$ and $e$ is odd and $a$ is non-square in $\mathbb{F}_{3^e}$.
 \item $X^6+4aX^2$ over $\mathbb{F}_{5^e}$, $e=2$ and $a=g^{4n+3}$, where $g$ is the root of the primitive polynomial $X^2+4X+2$ over $\mathbb{F}_5.$ 
 \item $X^{10}+2aX^6+a^2X^2$ over $\mathbb{F}_{5^e}$, $a\neq 4$ when $e=1$ and $-a$ is not a fourth power when $e\geq2$. 
\end{enumerate}
\end{thm}

\section*{Acknowledgments}
The research of Sartaj Ul Hasan is partially supported by Start-up Research Grant SRG/2019/000295 from the Science and Engineering Research Board, Government of India.

\appendix 
\section{The complete list of DO polynomials} \label{A1}

 The complete list of DO polynomials obtained from the polynomials $\mathfrak D_{k,m}$ over a finite field $\mathbb F_q$ of 
 odd characteristic $p$ is given as follows.
 \begin{enumerate}
  \item For $p=3$
 \begin{enumerate}
  \item When $m \equiv 0~(\mbox{mod}~3)$  \cite[Theorem 2.1]{CM}
  \begin{enumerate}
   \item $k=p^{\ell}$ then $X^{3^{n+\ell}(3^{\alpha}+1)}$ for non-negative integers $\alpha$, $n$ and $\ell$.
   \item $k=2p^{\ell}$ then $X^{3^{n+\ell}(3^{\alpha}+1)}$ for non-negative integers $\alpha$, $n$ and $\ell$.
   \item $k=4p^{\ell}$ then $X^{4\cdot3^{n+\ell}}+2aX^{2\cdot3^{n+\ell}}$ for non-negative integers $n$ and $\ell$.
   \item $k=5p^{\ell}$ then $X^{10\cdot3^{n+\ell}}+aX^{2\cdot3^{n+\ell+1}}+2a^{2}X^{2\cdot3^{n+\ell}}$ for non-negative integers $n$ and $\ell$.
  \end{enumerate}
  \item When $m \equiv 1~(\mbox{mod}~3)$ \cite[Theorem 3.1]{CM}
  \begin{enumerate}
   \item $k=1,2,4$ then $X^{3^n(3^{\alpha}+1)}$ for non-negative integers $\alpha$ and $n$.
   \item $k=3$ then $X^{3^{n+1}(3^{\alpha}+1)}+aX^{3^n(3^{\alpha}+1)}$ for non-negative integers $\alpha$ and $n$.
   \item $k=5$ then $X^{10\cdot3^{n}}+2aX^{2\cdot3^{n+1}}$ for non-negative integer $n$. 
   \item $k=6$ then $X^{2\cdot3^{n+1}}+aX^{4\cdot3^n}$ for non-negative integer $n$.
   \item $k=7$ then $X^{28\cdot3^{n}}+a^2X^{4\cdot3^{n+1}}+2a^3X^{4\cdot3^{n}}$ for non-negative integer $n$.
   \item $k=9$ then $X^{4\cdot3^{n+2}}+aX^{28\cdot3^{n}}+a^3X^{4\cdot3^{n+1}}+2a^4X^{4\cdot3^{n}}$ for non-negative integer $n$.
   \item $k=10$ then $X^{10\cdot3^{n}}+a^2X^{2\cdot3^{n+1}}+a^3X^{4\cdot3^{n}}$ for non-negative integer $n$.
   \item $k=12$ then $X^{4\cdot3^{n+1}}+aX^{10\cdot3^{n}}+a^4X^{4\cdot3^{n}}$ for non-negative integer $n$.
  \end{enumerate}
  \item When $m \equiv 2~(\mbox{mod}~3)$
  \begin{enumerate}
   \item $k=1,2$ then $X^{3^n(3^{\alpha}+1)}$ for non-negative integers $\alpha$ and $n$.
   \item $k=3$ then $X^{3^{n+1}(3^{\alpha}+1)}+2aX^{3^n(3^{\alpha}+1)}$ for non-negative integers $\alpha$ and $n$.
   \item $k=4$ then $X^{4\cdot3^{n}}+aX^{2\cdot3^{n}}$ for non-negative integer $n$. 
   \item $k=5$ then $X^{10\cdot3^{n}}+a^2X^{2\cdot3^n}$ for non-negative integer $n$.
   \item $k=6$ then $X^{2\cdot3^{n+1}}+2aX^{4\cdot3^{n}}$ for non-negative integer $n$.
   \item $k=9$ then $X^{4\cdot3^{n+2}}+2aX^{28\cdot3^{n}}+2a^3X^{4\cdot3^{n+1}}+a^4X^{4\cdot3^{n}}$ for non-negative integer $n$.
   \item $k=12$ then $X^{4\cdot3^{n+1}}+2aX^{10\cdot3^{n}}+a^3X^{2\cdot3^{n+1}}+2a^4X^{4\cdot3^{n}}$ for non-negative integer $n$.
  \end{enumerate}
 \end{enumerate}
 \item For $p=5$
  \begin{enumerate}
  \item When $m \equiv 0~(\mbox{mod}~5)$  \cite[Theorem 2.1]{CM}
  \begin{enumerate}
   \item $k=5^{\ell}$ then $X^{5^{n+\ell}(5^{\alpha}+1)}$ for non-negative integers $\alpha$, $n$ and $\ell$.
   \item $k=2.5^{\ell}$ then $X^{5^{n+\ell}(5^{\alpha}+1)}$ for non-negative integers $\alpha$, $n$ and $\ell$.
   \item $k=3.5^{\ell}$ then $X^{6\cdot5^{n+\ell}}+2aX^{2\cdot5^{n+\ell}}$ for non-negative integers $n$ and $\ell$.
  \end{enumerate}
  \item When $m \equiv 1~(\mbox{mod}~5)$ \cite[Theorem 3.1]{CM}
  \begin{enumerate}
   \item $k=1,2$ then $X^{5^n(5^{\alpha}+1)}$ for non-negative integers $\alpha$ and $n$.
   \item $k=3$ then $X^{6\cdot5^{n}}+3aX^{2\cdot5^n}$ for non-negative integers $\alpha$ and $n$.
   \item $k=5$ then $X^{2\cdot5^{n+1}}+aX^{6\cdot5^{n}}+3a^2X^{2\cdot5^{n}}$ for non-negative integer $n$.
   \item $k=9$ then $X^{6\cdot5^{n}}+a^2X^{2\cdot5^{n}}$ for non-negative integer $n$.
   \end{enumerate}
  \item When $m \equiv 2~(\mbox{mod}~5)$
  \begin{enumerate}
   \item $k=1,2$ then $X^{5^n(5^{\alpha}+1)}$ for non-negative integers $\alpha$ and $n$.
   \item $k=3$ then $X^{6\cdot5^{n}}+4aX^{2\cdot5^{n}}$ for non-negative integer $n$.
   \item $k=5$ then $X^{2\cdot5^{n+1}}+2aX^{6\cdot5^{n}}+a^2X^{2\cdot5^{n}}$ for non-negative integer $n$.
  \end{enumerate}
  \item When $m \equiv 3~(\mbox{mod}~5)$
  \begin{enumerate}
  \item $k=1,2,3$ then $X^{5^n(5^{\alpha}+1)}$ for non-negative integers $\alpha$ and $n$.
  \item $k=5$ then $X^{2\cdot5^{n+1}}+3aX^{6\cdot5^{n}}+4a^2X^{2\cdot5^{n}}$ for non-negative integer $n$.
  \end{enumerate}
  \item When $m \equiv 4~(\mbox{mod}~5)$
  \begin{enumerate}
   \item $k=1,2,4$ then $X^{5^n(5^{\alpha}+1)}$ for non-negative integers $\alpha$ and $n$.
   \item $k=3$ then $X^{6\cdot5^{n}}+aX^{2\cdot5^{n}}$ for non-negative integer $n$.
   \item $k=5$ then $X^{2\cdot5^{n+1}}+4aX^{6\cdot5^{n}}+2a^2X^{2\cdot5^{n}}$ for non-negative integer $n$.
  \end{enumerate}
 \end{enumerate}
 \item For $p>5$, we obtain only one DO polynomial given by $X^{p^{n}(p^{\alpha}+1)}$ for some non-negative integers $\alpha$ and $n$.
 \end{enumerate}

\end{document}